\documentclass[10pt, letterpaper, conference]{ieeeconf}  
\IEEEoverridecommandlockouts                              

\usepackage{graphicx, footnote, amssymb, mathtools}
 \usepackage{graphicx}
\usepackage{amsmath,epsfig,amssymb}
\usepackage{times,amsmath,epsfig,amssymb,graphicx,amsfonts}
\usepackage{color}
\usepackage{amsmath}
\usepackage{graphicx, footnote, amssymb, mathtools}
 \usepackage{graphicx}
\usepackage{amsmath,epsfig,amssymb}
\usepackage{times,amsmath,epsfig,amssymb,graphicx,amsfonts}
\usepackage{color}
\usepackage{graphicx, footnote, amssymb, mathtools}
 \usepackage{graphicx}
\usepackage{amsmath,epsfig,amssymb}
\usepackage{times,amsmath,epsfig,amssymb,graphicx,amsfonts}

\usepackage{color}
\usepackage{multirow}
\usepackage{psfrag}
\usepackage{epsfig}
\usepackage{subfig}
\usepackage{psfrag}
\usepackage{algorithm}
\usepackage{algorithmic}
 \newcommand{\bi}{\begin{itemize}}
\newcommand{\ei}{\end{itemize}}
\newcommand{\bal}{\begin{align}}
\newcommand{\eal}{\end{align}}

\newtheorem{theorem}{\textbf{Theorem}}
\newtheorem{lemma}{\textbf{Lemma}}

\newtheorem{definition}{\textbf{Definition}}

\newtheorem{remark}{\textbf{Remark}}
 
\newtheorem{proposition}{\textbf{Proposition}}
\newtheorem{assumption}{\textbf{Assumption}}

\title{\LARGE \bf Graph Balancing for Distributed Subgradient Methods over Directed Graphs}
\author{Ali Makhdoumi$^{*}$ and Asuman Ozdaglar$^{*}$
\thanks{$^{*}$
MIT,
Cambridge, MA 02139, 
Emails: makhdoum@mit.edu, asuman@mit.edu}
}

\begin{document}

\maketitle
\thispagestyle{empty}
\pagestyle{empty}
\begin{abstract}
We consider a multi agent optimization problem where a set of agents collectively solves a global 
optimization problem with the objective function given by the sum of locally known convex 
functions. We focus on the case when information exchange among agents takes place over a directed network and propose a distributed subgradient algorithm in which each agent performs local processing based on information obtained from his incoming neighbors. Our algorithm uses weight balancing to overcome the asymmetries caused by the directed communication network, i.e., agents scale their outgoing information with dynamically updated weights that converge to balancing weights of the graph. We show that both the objective function values and the consensus  violation, at the ergodic average of the estimates generated by the algorithm, converge with rate $O(\frac{\log T}{\sqrt{T}})$, where $T$ is the number of iterations. A special case of our algorithm provides a new distributed method to compute average consensus over directed graphs. 
\end{abstract}
\vspace{-.05 in}
\section{Introduction}
\vspace{- .05 in}
\subsection{Motivation}
Many of today's optimization problems in data science (including statistics, machine learning, and data mining) use distributed computation. Modern processors have access to thousands of data points much more than they can process. Therefore, there is a need to distribute data among different data centers and then process it in a decentralized way based on the information that is provided to them. Distributed computation can lead to gains in computational efficiency in statistical learning, as shown by a number of authors (\cite{dekel2012optimal, recht2011hogwild, agarwal2011distributed, duchi2014optimality}). The applications in statistical learning along with other applications in distributed control (see e.g. distributed sensor networks \cite{duarte2004vehicle}, coordination \cite{jadbabaie2003coordination}, and flow control  \cite{low1999optimization}) have motivated a flurry of research on distributed approaches for solving optimization problems where the objective function is the sum of local objective functions of agents (nodes) that are connected through a network (see  \cite{nedic2009distributed, wei2012distributed, wei20131, nedic2010constrained, nedic2007rate, lobel2011distributed, lobel2011distributed1, ram2009distributed, li2010competitive, rabbat2004distributed}).

Most of the existing algorithms 
assume information exchange over undirected networks where the connection between nodes are bidirectional, meaning that if node $i$ can send information to node $j$, then node $j$ can also send information to node $i$.\footnote{We use the terms graph and network interchangeably.} However, in many applications, the underlying graph is directed because nodes typically broadcast at different power levels and have varying interference and noise patterns, implying communication capability in one direction, but not the other. In this paper, we consider a multi-agent optimization problem where a set of agents collectively minimize the sum of locally known convex functions $f_i(x)$, using information exchange over a directed network.
\vspace{-.05 in}
 \subsection{Related Works and Contribution}
 Our paper is related to a large recent literature on distributed methods for solving multi agent optimization problems over networks. Much of this literature builds on the seminal works \cite{tsitsiklis1984problems, tsitsiklis1986distributed}, which proposed gradient methods that can parallelize computations across multiple processors. We summarize subgradient based distributed methods for solving multiagent optimization problems over undirected and directed graphs.  
\\ \textbf{Undirected Graphs: }A number of recent papers  proposed subgradient type distributed methods that use consensus or averaging mechanisms for aggregating information among the agents over an undirected network (see \cite{nedic2007rate, nedic2010constrained, lobel2011distributed, ram2010distributed, johansson2008subgradient, tsianos2012push, matei2011performance, nedic2011asynchronous,  zhu2012distributed,chen2012diffusion, duchi2012dual}). For convex local objective functions, these methods achieve $O(\frac{\log T}{\sqrt{T}})$ convergence rate, where $T$ is the number of iterations. The recent paper \cite{yuan2013convergence} adopts stronger assumptions on the local objective functions, i.e., they are convex with Lipschitz continuous gradients, and shows that  a distributed gradient method with averaging converges at rate $O(\frac{1}{T})$ to an error neighborhood (with the further assumption of strongly convex local objective functions, linear rate is achieved to an error neighborhood). Another contribution \cite{jakovetic2014fast} assumes local objective functions have continuous and bounded gradients and uses Nesterov's acceleration to design a distributed algorithm with rate $O(\frac{\log T}{T})$. The recent paper \cite{shi2014extra}, proposes a gradient-based distributed algorithm for convex objective function with Lipschitz continuous gradients and shows its convergence with rate $O(\frac{1}{T})$. 
\\ \textbf{Directed Graphs: }A few recent papers proposed and studied distributed subgradient methods over directed graphs \cite{tsianos2012push, nedic2013distributed, nedic2013stronglydistributed}. The key idea used in these papers is the incorporation of the {\it push-sum algorithm}, which is a distributed algorithm presented in \cite{kempe2003gossip} to obtain the average of initial values of nodes over a directed graph. Push-sum, while being an effective approach in obtaining averages over a directed graph, involves updates that include a nonlinear operation (division by weight estimates), which makes the analysis within a subgradient optimization method involved (see e.g. \cite{nedic2013distributed}). In this paper, we use an alternative approach based on weight-balancing to design a distributed subgradient algorithm over directed networks. The notion of weights that balance a directed graph was introduced in \cite{hooi1970class} and studied recently in a  number of papers with the goal of designing algorithms that enable their computation in a distributed manner (see \cite{gharesifard2014distributed,hadjicostis2012distributed, rikos2014distributed, priolo2013decentralized}). We combine such an algorithm for updating weights together with a distributed subgradient algorithm and show that these updates implemented simultaneously in the same time scale solves multiagent optimization problem over directed graphs. The update step for agent's estimates involves operators linear in each estimate, allowing the analysis to use techniques from time-varying non-homogeneous non-negative matrix theory \cite{seneta2006non}. Although we do not pursue this here, our algorithm can be generalized to work over time-varying directed graphs (see e.g. \cite{nedic2007rate, nedic2013distributed} for the analysis of distributed subgradient methods over time-varying graphs). Note that our algorithm does not require the local functions to have Lipschitz continuous gradient nor being strongly convex (the only requirement of our algorithm is that the functions are convex and not necessarily smooth). Indeed with these strong assumptions the rate of convergence can be improved as shown in  \cite{yuan2013convergence} and  \cite{nedic2013stronglydistributed} for algorithms over undirected graphs.

Our work also contributes to the vast literature on the consensus problem, where agents have a more specific goal of aligning their estimates (see \cite{jadbabaie2003coordination, olfati2004consensus, xiao2007distributed, olshevsky2009convergence} for consensus over undirected graphs and \cite{kempe2003gossip, dominguez2013distributed, priolo2014distributed} for consensus over directed graphs). In particular, a special case of our algorithm provides a new distributed method for computing average of initial values over a directed graph. Out contribution here is most closely related to \cite{priolo2014distributed} and \cite{dominguez2013distributed}, which proposed distributed algorithms for average consensus over  directed graphs using balancing weights. Reference \cite{priolo2014distributed} builds on earlier work \cite{priolo2013decentralized}, which presented a distributed algorithm for computing balancing node weights based on approximating the left eigenvector associated with the zero eigenvalue of the Laplacian matrix of the underlying directed network. In \cite{priolo2014distributed}, the authors used this algorithm for updating weights in the same time scale as the update for estimates and showed convergence to average of initial values. Reference \cite{dominguez2013distributed} provides a similar algorithm for average consensus based on an earlier work \cite{hadjicostis2012distributed}, which proposes an update rule for computing balancing edge weights. 

\vspace{-.05 in}
\subsection{Outline}
The organization of paper is as follows. In Section \ref{sec:formulation}, we give the problem formulation and present the distributed algorithm. In Section \ref{sec:convergenceweights}, we show the convergence of weights to balancing weights of the graph. In Section \ref{sec:convergencealgorithm}, we consider the sequence of ergodic (time) averages of the estimates generated by our algorithm. We then show optimality convergence of our algorithm and establish $O(\frac{\log T}{ \sqrt{T}})$ rate of convergence. Finally, in Section \ref{sec:numerical}, we provide numerical results that illustrate the performance of our algorithm, which leads to concluding remarks in Section \ref{sec:conclusion}. 
\vspace{-.05 in}
\subsection{Basic Notations }
A vector $\mathbf{x}$ is viewed as a column vector. For a matrix $A$, we write 
$[A]_i$ to denote the $i$th column of
matrix $A$, $[A]^i$ to denote the $i$th row of matrix $A$, and $A_{ij}$ to denote the entry at $i$th row and $j$th column. For a vector $\mathbf{x}$, $x_i$  denotes the $i$th component 
of the vector. We show the $L_1$ norm of a vector $\mathbf{x}\in \mathbb{R}^n$ with $||\mathbf{x}||_1 \triangleq {\sum_{i=1}^n |x_i|}$. We also let $||\mathbf{x}||_{\infty}=\max_{i} |x_i|$. For a set $S$, $|S|$ denotes the number of elements of $S$. 
\vspace{- .05 in}
\section{Problem Setup and Algorithm}\label{sec:formulation}
\subsection{Formulation}
We consider a set of agents (nodes) $V=\{1, \dots, n\}$ connected through a directed graph $G=(V, E)$ where $E \subset V \times V$ is the set of directed edges, i.e., $(i, j) \in E$ represents a directed edge from agent $i$ to agent $j$. We denote the in-neighbors and out-neighbors of an agent $i$ by $N^{\text{in}}(i)=\{j \in V ~ |~ (j , i) \in E\}$ and $N^{\text{out}}(i)=\{j \in V ~ |~ (i , j) \in E\}$, respectively. We also use  $d_i^{\text{in}}=|N^{\text{in}}(i)|$ and $d_i^{\text{out}}=|N^{\text{out}}(i)|$ to denote the number of in-neighbors and out-neighbors. 

We consider the following optimization problem 
\begin{align}\label{eq:optfomulation}
\min_{x \in \mathbb{R}} \sum_{i=1}^n f_i(x),
\end{align}
where $f_i: \mathbb{R} \to \mathbb{R}$ is a convex function (possibly non-smooth), known to agent $i$ only. The goal is to solve \eqref{eq:optfomulation} using an algorithm that involves each agent performing computations based on his {\it local objective function}  $f_i$ and exchanging information over the directed graph (i.e., receiving information from his in-neighbors and sending the outcome of his computation to his out-neighbors).

We adopt the following standard assumption on the underlying graph $G$. 
\begin{assumption}[Strongly connected graph]\label{assump:strongconnectivity}
\textup{
The graph $G=(V, E)$ is strongly connected, i.e., for all nodes $i, j \in V$, there exists a directed path from $i$ to $j$. 
}
\end{assumption}
This assumption ensures that every node $i$ receives information from every other node $j$ in the graph. 
\vspace{-.05 in}
\subsection{Algorithm}
Our algorithm generalizes the distributed subgradient algorithm presented in \cite{nedic2009distributed} to allow its implementation over directed graphs. The algorithm in \cite{nedic2009distributed} requires some form of symmetry in information exchange between pairs of nodes (essentially that node $j$'s information in node $i$'s update is scaled with the same weight used in scaling node $i$'s information in node $j$'s update). This kind of symmetry does not exist in directed graphs with directed communication over edges. Hence, a direct application of the algorithm given in \cite{nedic2009distributed} will lead to information from high out-degree nodes to be disproportionately represented in the estimate formed by agents. To alleviate this, our proposed algorithm scales outgoing information from each node with time-varying weights which in the limit ensures the incoming and the outgoing information of a node to be balanced. The following definition introduces node weights that satisfy this balancing requirement. The notion of node weights that balance a directed graph was proposed in  \cite{hooi1970class} and used recently for deriving a Lyapunov function for convergence analysis
of average-consensus \cite{olfati2007consensus}, \cite{olfati2004consensus}, consensus on general
functions \cite{cortes2008distributed}, design of stable flocking algorithms \cite{lee2007stable}, and traffic-flow problems \cite{hooi1970class}. 

\begin{definition}[Balancing weights]
\textup{
The node weights $w_i$ for $i \in V$ balance a directed graph $G$ if for any $i$, we have 
\begin{align*}
w_i d_i^{\text{out}} =\sum_{j \in N^{\text{in}}(i)} w_j.
\end{align*}
This definition ensures that the total weight outgoing from node $i$ (measured by $w_i d_i^{\text{out}}$) is equal to the total weight incoming to node $i$ (measured by $\sum_{j \in N^{\text{in}}(i)} w_j$), hence the term  balancing weights. 
}
\end{definition}
We next describe our algorithm. Let $x_i(t) \in \mathbb{R}$ denote the estimate of agent $i$ at time $t$ for the optimal solution of \eqref{eq:optfomulation}. Each agent $i$ starts from arbitrary value $x_i(0) \in \mathbb{R}$ and weight $w_i(0) \in \mathbb{R}$. At time $t$, agent $i$ updates its estimate $x_i(t)$ as 
\begin{align}\label{eq:estimateupdate}
 x_i(t+1) = & x_i(t) \left( 1- w_i(t) d_i^{\text{out}} \right) \nonumber \\& +  \left( \sum_{j \in N^{\text{in}}(i)} w_j(t)  x_j(t) \right)   - \alpha(t) g_i(t),
\end{align}
where $g_i(t)$ is a subgradient of $f_i$ at $x_i(t)$, i.e., $g_i(t) \in \partial f_i(x_i(t))$, $\alpha(t)$ is a step size sequence and $w_i(t)$ is a scalar weight. Each agent $i$ linearly combines the estimates of his incoming neighbors and his own estimate and takes a step along the  negative subgradient of his local objective function. This is followed by the following weight update at node $i$
\begin{align}\label{eq:weightupdate}
w_{i}(t+1) = \frac{1}{2} w_{i}(t) + \frac{1}{d_i^{\text{out}}} \sum_{j \in N^{\text{in}}(i)} \frac{1}{2} w_{j}(t).
\end{align}
In order to understand this update rule, note that if the sequence of weights $\{w_i(t)\}_t$, $i=1, \dots, n$ converges, it follows from \eqref{eq:weightupdate} that the limiting weights balance the graph. With these balancing weights, update \eqref{eq:estimateupdate} ensures that in the limit the incoming information and the outgoing information of node $i$ is properly scaled.

In principle, update \eqref{eq:estimateupdate} is similar to the algorithm considered in \cite{nedic2009distributed} for updating estimates, i.e., each agent updates his estimate by linearly combining estimates of neighbors together with a local optimal step. However, note that the weight matrix, i.e., the matrix that contains the scalars that multiply the estimates of the agents, is not doubly stochastic. We will show in Sections \ref{sec:convergenceweights}-\ref{sec:convergencealgorithm} that it is column stochastic and becomes doubly stochastic only in the limit and our analysis shows this property suffices to guarantee that the estimates (obtained by \eqref{eq:estimateupdate}) converge to the optimal solution of problem \eqref{eq:optfomulation}.


\vspace{-.05 in}
\section{Convergence of Weights} \label{sec:convergenceweights}
In this section, we show that the sequence of weights generated by the update \eqref{eq:weightupdate} converges to balancing weights for $G$. 
Using the notation $\mathbf{w}(t)=[w_1(t), \dots, w_n(t)]'$,  we can write the weight updates of our algorithm more compactly as 

\begin{align}
\mathbf{w}(t+1) = \frac{1}{2} (I+ D^{-1}  A ) \mathbf{w}(t),
\end{align}
where $D=\text{diag}(d_1^{\text{out}}, \dots, d_n^{\text{out}})$ and $A$ is the adjacency matrix of the directed graph defined as $A_{ij}=1$ for all $i, j \in N^{\text{in}}(i)$. We let $P= \frac{1}{2} (I+ D^{-1}  A )$ and rewrite the weight updates as  
\begin{align*}
\mathbf{w}(t+1) = P \mathbf{w}(t).
\end{align*}
Next, we show the weight sequence generated by the update \eqref{eq:weightupdate} converges to balancing weights. 
\begin{lemma} \label{lem:convergenceofweights}
Under Assumption \ref{assump:strongconnectivity}, the sequence $\mathbf{w}(t)$ converges to $\mathbf{w}$ and $\mathbf{w}=[ w_1, \dots, w_n ]'$ balances the graph. 
\end{lemma}
\begin{proof}
We first show that $P$ is primitive.\footnote{A matrix $P$ is primitive if there exists $n \in \mathbb{N}$ such that all entries of $P^n$ are positive.} Note that all entries of $P$ are non-negative, the diagonal entries of $P$ are positive, and the underlying graph is strongly connected. The result follows from these facts. We next show the spectral radius of $P$, $\rho(P)$, is one. We define an auxiliary matrix $\bar{P}=  \frac{1}{2} (I+ A D^{-1} )$ and show it is column stochastic. This holds because the summation of the entries on the $j$th column of $\bar{P}$ is $\frac{1}{2} + \sum_{i~:~ j \in N^{\text{in}}(i)} \frac{1}{2} \frac{1}{d_j^{\text{out}}}=1$. Since $\bar{P}$ is column stochastic, we have $\rho(\bar{P})=1$. Therefore, we have $ 1=\rho(\bar{P})= \rho(\bar{P} D D^{-1} ) = \rho(D^{-1} \bar{P} D )
 = \rho(\frac{1}{2}( I + D^{-1} A D^{-1} D))= \rho(P)$,  
where the equality $ \rho(\bar{P} D D^{-1} ) = \rho(D^{-1} \bar{P} D )$ holds because the set of eigenvalues of $AB$ is the the same as the set of eigenvalues of $BA$ for two arbitrary matrices $A$ and $B$. 
Finally, since $P$ is a primitive matrix with $\rho(P)=1$, using Perron-Frobenius Theorem, the limit $\lim_{t \to \infty} P^t$ exists and as a result $\mathbf{w}(t)$ converges to some $\mathbf{w}$ that satisfies $\mathbf{w}= P \mathbf{w}$. Using the definition of $P$, this yields $w_i= \frac{1}{2} w_i + \sum_{j: ~ j \in N^{\text{in}}(i)} \frac{1}{2} \frac{1}{d_i^{\text{out}}} w_j$, 
which results in $d_i^{\text{out}} \frac{1}{2} w_i = \sum_{j \in N^{\text{in}}(i)} \frac{1}{2} w_j$. 
Therefore, $\{\mathbf{w}(t)\}_{t=0}^{\infty}$ converges to weights that balance the graph.

\end{proof}

The generated  weight sequence $\mathbf{w}(t)$ depends on the initial value of $\mathbf{w}(0)$. We consider the sequence update \eqref{eq:estimateupdate} and note that in order to guarantee the contribution of agent $i$'s estimate at time $t$ in his estimate at time $t+1$ is positive, we need  the coefficient of $x_i(t)$ which is $1- w_i(t) d_i^{\text{out}}$ to be positive at each iteration. In the next lemma we show that by choosing a small $\mathbf{w}(0)$, we can guarantee this. Let $D$ denote the diameter and $d^*=\max_i d_i^{\text{out}}$ denote the maximum out-degree of the graph $G$. 
\begin{lemma}\label{lem:smallwbeta}
If for all $i=1, \dots, n$, $w_i(0) \le ({1}/{{d^*})^{2D+1}}$, then $w_i(t) d_i^{\text{out}}<1$ for all $i$ and $t \ge 0$. 
\end{lemma}
\begin{proof}
Let $\mathbf{u}$ and $\mathbf{v}$ be left and right eigenvectors of $P$ corresponding to eigenvalue one; i.e. $\mathbf{u}' P= \mathbf{u}'$ and $P \mathbf{v} =\mathbf{v}$.  From Perron-Frobenius Theorem, we have $\lim_{t \to \infty} P^t = {\mathbf{v} \mathbf{u}'}$, where $\mathbf{u} > 0$, $\mathbf{v} > 0$, $\sum_{i=1}^n v_i=1$, and $\sum_{i=1}^n v_i u_i=1$. Next, we will bound the entries of $\mathbf{v}$.
Let $v_M= \max_{i} v_i$ and $v_m= \min_i v_i$ denote the maximum and minimum entries among $v_i$'s. The $i$-th equation of $P\mathbf{v}=\mathbf{v}$ can be written as 
\begin{align}\label{eq:templemboundperronvector}
v_i= \sum_{k \in N^{\text{in}}(i)}\frac{1}{d_i^{\text{out}}-1} v_k ~~\text{ for all } k \in N^{\text{in}}(i).
\end{align}
Using \eqref{eq:templemboundperronvector}, for any $i$ we can bound $v_i$ by a factor of $v_k$ for $k \in N^{\text{in}}(i)$ as follows
\begin{align}\label{eq:templemboundperronvector2}
v_i \ge \frac{1}{d^*} v_k ~~\text{ for all } k \in N^{\text{in}}(i), 
\end{align}
where $d^* = \max_i d_i^{\text{out}}$. Using the above inequality, for any $k$, we obtain that $ v_i \ge \left( \frac{1}{d^*} \right)^{D} u_M, \text{ for all } i$, where $D$ is the the diameter of the graph. Using this inequality along with $\sum_{i=1}^n v_i=1$, we obtain \[1 = \sum_{i=1}^n v_i \ge v_M n \left( \frac{1}{d^*} \right)^{D},\] 
which yields to $v_M \le \frac{1}{{n}} \left( {d^*} \right)^D$. Next, we find a lower bound on the value of $v_m$. Again, using \eqref{eq:templemboundperronvector2}, we obtain that $v_m \ge \left( \frac{1}{d^*} \right)^D v_i, \text{ for all } i$.  
Using this inequality along with $\sum_{i=1}^n v_i=1$, we obtain $1 = \sum_{i=1}^n v_i \le v_m n {d^*}^{D}$, 
which yields to $v_m \ge \frac{1}{{n}} \left( \frac{1}{d^*} \right)^D$. 
\\Next, we will use the derived bounds on $v_m$ and $v_M$ in order to bound the summation of entries of each row of $P^t$.  Since $P \mathbf{v}=\mathbf{v}$, for any $t$, we have $P^t \mathbf{v}=\mathbf{v}$. Because the entries of $P^t$ are non-negative, for any $i$ we have 
\begin{align*}
v_M=\max_{1 \le i \le n} v_i \ge v_i = [P^t v]_i= \sum_{j=1}^n [P^t]_{ij} v_j \ge v_m \sum_{j=1}^n [P^t]_{ij}.
\end{align*}
Plugging in the bounds on $v_m$ and $v_M$, we obtain 
\[\sum_{j=1}^n [P^t]_{ij} \le \frac{v_M}{v_m} \le \frac{\frac{1}{n}{d^*}^D}{\frac{1}{n} (\frac{1}{d^*})^D}= {d^*}^{2D}.\] 
Therefore, if we let $w_i(0) < \frac{1}{{d^*}^{(2D+1)}}$ for all $i \in V$, then 
\begin{align*}
d_i^{\text{out}} w_i(t) =  d_i^{\text{out}} \sum_{j=1}^n [P^t]_{ij} w_j(0) \le d_i^{\text{out}}{d^*}^{2D} ||\mathbf{w}(0)||_{\infty} < 1,
\end{align*} 
and we obtain $1- w_i(t) d_i^{\text{out}} > 0$ for all $t \ge 0,$ and $i \in V$. 

\end{proof}



\vspace{-.1 in}
\section{Convergence of Algorithm}\label{sec:convergencealgorithm}
\subsection{Preliminary Results}
We can write the updates of our algorithm \eqref{eq:estimateupdate} in a compact form as
\begin{align}\label{eq:updatecompact}
\mathbf{x}(t+1)= Q(t) \mathbf{x}(t) - \alpha(t) \mathbf{g}(t),
\end{align}
where 
\begin{align*}
\mathbf{g}(t)=[g_1(t), \dots, g_n(t)]',
\end{align*}
\begin{align*}
\mathbf{x}(t)=[x_1(t), \dots, x_n(t)]', 
\end{align*}
and $Q(t)$ is a matrix such that $[Q(t)]_{ii}= 1- w_i(t) d_i^{\text{out}}$ and $[Q(t)]_{ij}= - w_j(t)$ for any $i$ and $j \in N^{\text{in}}(i)$.

Using \eqref{eq:updatecompact} recursively, we can relate the estimate at time $t+1$ to initial estimates and the intermediate matrices and subgradients as follows 
\begin{align}\label{eq:updatecompact2}
\mathbf{x}(t+1) & =  \left[ Q(t) Q(t-1) \dots Q(0) \right] \mathbf{x}(0)  \nonumber \\
& - \sum_{s=0}^{t-1} \alpha(s) \left[ Q(t) \dots Q(s+1) \right]  \mathbf{g}(s)  - \alpha(t) \mathbf{g}(t).
\end{align}
This motivates the following definition 
\begin{align*}
\Phi(t:s)= Q(t) \dots Q(s), ~~ \text{ for } t \ge s, 
\end{align*}
with the convention that $\Phi(t:t+1)=I$. Using this definition \eqref{eq:updatecompact2} can be written as 
\begin{align}\label{eq:tempxiupdtae}
\mathbf{x}(t+1) = \Phi(t:0) \mathbf{x}(0)  - \sum_{s=0}^t  \Phi(t:s+1) \mathbf{g}(s)  \alpha(s).
\end{align}

Next, we show that for any $t$ the matrix $Q(t)$ is column stochastic and becomes doubly stochastic only in the limit (i.e., $Q$ is doubly stochastic, where $\lim_{t \to \infty} Q(t)=Q$) and our analysis shows this suffices to guarantee the entries of $\Phi(t:s)$ converge to $\frac{1}{n}$ exponentially fast as $t-s$ grows large. In the analysis that follows, we will use Theorem 4.14 and Theorem 4.19 of \cite{seneta2006non}. 
\begin{proposition}[\cite{seneta2006non}]\label{pro:senata}
\begin{enumerate}
\item Let $Q(t)$ be a sequence of column stochastic matrices such that $\lim_{t \to \infty} Q(t)= Q$ (entry-wise),  where $Q$ is a primitive matrix. Then, for any $s$, we have that $\lim_{t \to \infty} [\Phi(t:s)]_{ij}= p_j$, where $\mathbf{p}$ is the unique probability vector (non-negative with summation equal to one) such that $Q\mathbf{p}=\mathbf{p}$. 
\item Let $Q(t)$ be a sequence of column stochastic primitive  matrices such that $\min^{+}_{ij} [Q(t)]_{ij} \ge \gamma$, where $\min^{+}$ denotes the minimum among positive entries.  Then there exist a probability vector $\mathbf{p}(s)$ such that 
\begin{align*}
|[\Phi(t:s)]_{ij}- p_j(s)| \le \frac{1}{1-\gamma}  (1- \gamma)^{t-s+1}.
\end{align*}
\end{enumerate}
\end{proposition}



\vspace{-.03 in}
\begin{lemma}\label{lem:matrixconvergence}
There exist constants $C$ and $\lambda$ such that for any $t \ge s$ and $i, j \in V$, we have 
\begin{align}
|[\Phi(t:s)]_{ij}- \frac{1}{n}| \le C \lambda^{(t-s+1)}. 
\end{align}
\end{lemma}
\begin{proof} 
 Using Lemma \ref{lem:smallwbeta}, since for any $i$,  $d_i^{\text{out}}  w_i(t) < 1$, all the entries of matrices $Q(t)$ are non-negative for $t \ge 0$. Furthermore, for any $t$ the matrix $Q(t)$ is column stochastic because for any $i$, $\sum_{i=1}^n Q_{ij}(t)= 1-  w_j(t) d_j^{\text{out}} + \sum_{i:~ j \in N^{\text{in}}(i)}  w_j(t)  = 1$. 
Using Lemma \ref{lem:convergenceofweights},  $\lim_{t \to \infty} Q(t)= Q$, where $Q_{ii}=1- w_i d_i^{\text{out}}$ and $Q_{ij}= w_i$ for $j \in N^{\text{in}}(i)$. The matrix $Q$ is row stochastic as well because for  any $j$, $\sum_{i=1}^n Q_{ij}= 1- d_j^{\text{out}}  w_j + \sum_{i: ~ j \in N^{\text{in}}(i)}  w_j= 1$.  
Since $Q$ is both column and row stochastic the unique probability vector $\mathbf{p}$ for which $Q\mathbf{p}=\mathbf{p}$ holds, is a vector with all entries equal to $\frac{1}{n}$. Also note that since all entries of $Q$ are non-negative, the diagonal entries of $Q$ are positive, and the underlying graph is strongly connected, we know that the matrix $Q$ is primitive (with the same argument $Q(t)$ is primitive for any $t$). 
\\We now have all the conditions to use the first part of Proposition \ref{pro:senata}, to obtain  $\lim_{t \to \infty} [\Phi(s:t)]_{ij}=\frac{1}{n}$, for all $i, j\in V$. 
\\Next, we use the second part of Proposition \ref{pro:senata} in order to establish the exponential rate of convergence. 
 Let $\delta= \min^{+}_{ij} Q_{ij}$. Since $\lim_{t \to \infty} Q(t)= Q$, there exists $t_0$ such that for any $t \ge t_0$, we have that $|| Q(t)- Q||_{\infty} \le \frac{\delta}{2}$. Therefore, for $t \ge t_0$, we have that $\min^{+}[Q(t)]_{ij} \ge \frac{\delta}{2}$. We let   $\gamma= \min\{\frac{\delta}{2}, \min^{+}_{t \le t_0, ~ i, j} [Q(t)]_{ij}\}$ to obtain $\min^{+}_{ij} [Q(t)]_{ij} \ge \gamma$, for any $t$. Now we can use the second part of Proposition \ref{pro:senata} to obtain that for any $s$, 
\begin{align*}
|[\Phi(t:s)]_{ij}- \frac{1}{n}| \le \frac{1}{1-\gamma}  (1- \gamma)^{t-s+1}.
\end{align*}
This completes the proof. 
\end{proof}
\begin{remark}
\textup{
When $f_i(x)=0$ for all $i \in V$ and all $x\in \mathbb{R}$ (implying $g_i(t)=0$, for all $i \in V$ and all $t \ge 0$), the update step \eqref{eq:estimateupdate} simplifies to 
\begin{align}\label{eq:consensuupdate}
 x_i(t+1) =  x_i(t) \left( 1- w_i(t) d_i^{\text{out}} \right)  +  \sum_{j \in N^{\text{in}}(i)} w_j(t)  x_j(t) .
\end{align}
For this special case combining \eqref{eq:tempxiupdtae} with Lemma \ref{lem:matrixconvergence} shows that the estimates $x_i(t)$ generated by algorithm \eqref{eq:estimateupdate}-\eqref{eq:weightupdate} converges to the average of the initial values, i.e., $\lim_{t \to \infty} x_i(t)= \frac{1}{n} \sum_{i=1}^n x_{i}(0)$, 
and the rate of convergence is exponential. Hence, our algorithm provides a new distributed method for computing average of initial values over a directed graph (see \cite{priolo2014distributed, dominguez2013distributed} for similar algorithms). 
}
\end{remark}
\vspace{-.05 in}
\subsection{Consensus in Estimates}
We will first show that under some mild assumptions on the step size sequence, the disagreement between estimates of agents goes to zero, i.e., $\lim_{t \to \infty} |x_i(t)-x_j(t)|=0$, for all $i, j \in V$. To that end, define an auxiliary sequence as 
\begin{align}\label{eq:tempytrecur}
y(t+1)= y(t)-\frac{\alpha(t)}{n}\sum_{i=1}^n g_i(t).
\end{align}
We let $y(0)= \frac{1}{n} \sum_{i=1}^n x_i(0)$ to obtain 
\begin{align}\label{eq:tempyt}
y(t)= \frac{1}{n} \sum_{i=1}^n x_i(0) - \frac{1}{n} \sum_{s=0}^{t-1} \alpha(s) \sum_{i=1}^n g_i(s).
\end{align}
Using the compact form \eqref{eq:tempxiupdtae}, the update of node $i$ at time $t+1$ can be written as 
\begin{align}\label{eq:tempxiupdtae2}
x_i(t+1) = \left[ \Phi(t:0) \mathbf{x}(0) \right]_i - \sum_{s=0}^t \left[ \Phi(t:s+1) \mathbf{g}(s) \right]_i \alpha(s). 
\end{align}
We establish the convergence rate of our algorithm using the ergodic average of the sequence $\{x_i(t)\}$ generated by algorithm \eqref{eq:estimateupdate}-\eqref{eq:weightupdate}, defined as  
\begin{align}
\hat{x}_i(T) = \frac{1}{\sum_{t=0}^T \alpha(t)} \sum_{t=0}^T \alpha(t) x_i(t),
\end{align}
and 
\begin{align}
{\hat{y}}(T)= \frac{1}{\sum_{t=0}^T \alpha(t)} \sum_{t=0}^T \alpha(t) {y}(t).
\end{align}
\begin{assumption}
\textup{
Suppose that the functions $f_i$ have bounded subgradient, i.e., there exists $L$ such that for any $i \in V$ and $x$
\begin{align*}
|g_i| \le L, ~ \text{ for any }  g_i \in \partial f_i(x) \text{ and any } x. 
\end{align*}
}
\end{assumption}


\begin{lemma} \label{lem:feasibility}
Let the sequence $\{\alpha(t)\}_{t=0}^{\infty}$ be non-negative. For all $i \in V$, we have  
\begin{align}
|x_i(t)- y(t)| \le C \lambda^{t} ||\mathbf{x}(0)||_1 + n L \sum_{s=0}^{t-1} C \lambda^{t-s-1} \alpha(s).
\end{align}
Moreover, if the sequence $\{\alpha(t)\}_{t=0}^{\infty}$ is non-increasing and converges to zero, then for any $i$ and $j$, 
\begin{align*}
\lim_{t \to \infty} |x_i(t) - x_j(t)|=0.
\end{align*} 
\end{lemma}
\begin{proof}

Using \eqref{eq:tempxiupdtae2}, \eqref{eq:tempyt}, and H\"{o}lder's inequality, we have 
\begin{align*}
& |x_i(t)- y(t)| = |\left[ \Phi(t-1:0) \mathbf{x}(0) \right]_i  \\
& -  \sum_{s=0}^{t-1} \left[ \Phi(t-1:s+1) \mathbf{g}(s) \right]_i \alpha(s) \\
& - \frac{1}{n} \sum_{i=1}^n x_i(0) + \frac{1}{n} \sum_{s=0}^{t-1} \alpha(s) \sum_{i=1}^n g_i(s)| \\
& \le  \max_{j}|[\Phi(t-1:0)]_{ij}- \frac{1}{n}| \times ||\mathbf{x}(0)||_1 \\
& + n \sum_{s=0}^{t-1} \alpha(s) L \max_{j}|[\Phi(t-1:s+1)]_{ij}- \frac{1}{n}|\\
& \le C \lambda^{t} ||\mathbf{x}(0)||_1 + n L \sum_{s=0}^{t-1} C \lambda^{t-1-s} \alpha(s),
\end{align*}
where we used Lemma \ref{lem:matrixconvergence} to obtain the last inequality. For a given $\epsilon$, let $t_0$ be such that for $t \ge t_0$, $\alpha(t) \le \frac{\epsilon (1- \lambda)}{2}$ to obtain
\begin{align*}
&\sum_{s=0}^{t-1} \alpha(s) \lambda^{t-1 - s} = \sum_{s=0}^{t_0} \alpha(s) \lambda^{t-1 - s} + \sum_{s=t_0+1}^{t-1} \alpha(s) \lambda^{t-1 - s} \\ 
& \le \left( \max_{ 0 \le s \le t_0} \alpha(s) \right) \lambda^ {-t_0} \lambda^{t-1} \sum_{s=0}^{t_0} \lambda^{s} \\
& + \frac{\epsilon (1- \lambda)}{2} \sum_{s=t_0+1}^{t-1} \lambda^{t-1 - s} \le  \alpha(0) \lambda^ {-t_0} \frac{1}{1-\lambda}  \lambda^{t-1}  + \frac{\epsilon}{2}.
\end{align*}
Using this relation, we have  
$\lim_{t \to \infty}\sum_{s=0}^{t-1} \alpha(s) \lambda^{t-1 - s}  \le  \frac{\epsilon}{2}$. 
Therefore, for any $i, j$, we obtain 
\begin{align*}
\lim_{t \to \infty} |x_i(t)-x_j(t)| \le \left( \lim_{t \to \infty} 2 \lambda^t ||\mathbf{x}(0)||_1 \right)+ n LC \epsilon= nLC \epsilon.
\end{align*}
Since $\epsilon$ is arbitrary, by taking $\epsilon \to 0$, we conclude $\lim_{t \to \infty} |x_i(t)- x_j(t)|=0$. 
\end{proof}
Lemma \ref{lem:feasibility} shows the convergence of the sequences $\{x_i(t)-x_j(t)\}_{t=0}^{\infty}$ to zero. Next, we bound the difference between the ergodic average sequences $\{\hat{x}_i(t)\}_{t=0}^{\infty}$ and $\{\hat{y}(t)\}_{t=0}^{\infty}$. 
\begin{theorem}\label{th:feasibility}
The ergodic average of the sequence generated by algorithm \eqref{eq:estimateupdate}-\eqref{eq:weightupdate}, satisfies
\begin{align}\label{eq:feasibility}
& |\hat{x}_i(T) - \hat{y}(T)|  \nonumber  \le \frac{2}{\sum_{t=0}^T \alpha(t)}\\
&  \sum_{t=0}^T \alpha(t)  \left( C  \lambda^{t} ||\mathbf{x}(0)||_1 + L n \sum_{s=0}^{t-1} C \lambda^{t-1-s} \alpha(s) \right)
\end{align}
\end{theorem}
\begin{proof}
Using convexity of norm, for any $i$, we have  
\begin{align*}
& |\hat{x}_i(T) - \hat{y}(T)| \le \frac{1}{\sum_{t=0}^T \alpha(t)} \sum_{t=0}^T \alpha(t) |x_i(t) - y(t)|.
\end{align*}
We next use Lemma \ref{lem:feasibility} to upper bound each term on the right hand side of the previous relation, which yields to
\begin{align*}
& |\hat{x}_i(T) - \hat{y}(T)| \le\\& \frac{1}{\sum_{t=0}^T \alpha(t)} \sum_{t=0}^T \alpha(t)  \left( C  \lambda^{t} ||\mathbf{x}(0)||_1 + L n \sum_{s=0}^{t-1} C \lambda^{t-1-s} \alpha(s) \right).
\end{align*}
\end{proof}
\vspace{- .1 in}
\subsection{Optimality Convergence}
Let $F(\mathbf{x}(t))=\sum_{i=1}^n f_i(x_i(t))$ and $F(\mathbf{x}^*)=\sum_{i=1}^n f_i(x^*)$, where $\mathbf{x}^*=(x^*, \dots, x^*) \in \mathbb{R}^n$ is an optimal solution of \eqref{eq:optfomulation}. Also let $\hat{\mathbf{x}}(T)=(\hat{x}_1(T), \dots, \hat{x}_n(T))$. 
\begin{lemma}\label{lem:optimality1}
 The ergodic average of the sequence generated by algorithm \eqref{eq:estimateupdate}-\eqref{eq:weightupdate}, satisfies 
\begin{align}
& F(\hat{\mathbf{x}}(T))- F(\hat{\mathbf{y}}(T))  \le   L \frac{1}{\sum_{t=0}^T \alpha(t)}  \nonumber \\
& \left( \sum_{t=0}^T \alpha(t) \left( C \lambda^{t} ||\mathbf{x}(0)||_1 + L n  \sum_{s=0}^{t-1} C \lambda^{t-s} \alpha(s) \right) \right)
\end{align}
\end{lemma}
\begin{proof}
Using bounded subgradient assumption and convexity of norm function, we obtain that
\begin{align*}
& F(\hat{\mathbf{x}}(T))- F(\hat{\mathbf{y}}(T))  \le \sum_{j=1}^n g_j(\hat{x}_j(T)) (\hat{x}_j(T)- \hat{y}(T)) \\
& \le \sum_{j=1}^n L |\hat{x}_j(T)- \hat{y}(T)| \le L \frac{1}{\sum_{t=0}^T \alpha(t)} \\
& \left( \sum_{t=0}^T \alpha(t) \left( C \lambda^{t} ||\mathbf{x}(0)||_1 + L  n \sum_{s=0}^{t-1} C \lambda^{t-s} \alpha(s) \right) \right),
\end{align*}
where we used Theorem \ref{th:feasibility} to obtain the last inequality.
\end{proof}
\begin{lemma}\label{lem:optimality2} 
The ergodic average of the sequence generated by algorithm \eqref{eq:estimateupdate}-\eqref{eq:weightupdate}, satisfies 
\begin{align}\label{eq:lem:optimality2}
& F(\hat{\mathbf{y}}(T))- F(\mathbf{x}^*) \le \frac{n}{2 \sum_{t=0}^T \alpha(t)} \left( y(0)-x^* \right)^2 \nonumber \\
& + \frac{n L^2}{2 \sum_{t=0}^T \alpha(t)} \sum_{t=0}^T {\alpha(t)^2} \nonumber\\
& + \frac{2 L}{ \sum_{t=0}^T \alpha(t)} \sum_{t=0}^T  \alpha(t) \sum_{i=1}^n |y(t)-x_i(t)|. 
\end{align}
\end{lemma}
\begin{proof}
Using \eqref{eq:tempytrecur}, we have that 
\begin{align}\label{eq:lemtemp1}
& \left(  y(t+1)-x^* \right)^2 = \left( y(t)-x^* \right)^2 + \frac{\alpha(t)^2}{n^2} \left( \sum_{i=1}^n g_i(t) \right)^2 \nonumber \\
& - \frac{2 \alpha(t)}{n} \sum_{i=1}^n g_i(t) (y(t)-x^*) .
\end{align}
Since $|g_i(t)| \le L$, we also have that 
\begin{align*}
& g_i(t) (y(t)-x^*)  = g_i(t) (y(t)-x_i(t)) + g_i(t) (x_i(t)-x^*) \\
& \ge g_i(t) (y(t)-x_i(t)) + f_i(x_i(t)) - f_i(x^*) \\
& \ge -L |y(t)-x_i(t)| \\
& + f_i(x_i(t)) - f_i(y(t))
+ f_i(y(t)) - f_i(x^*).
\end{align*}
Using $f_i(x_i(t))- f_i(y(t)) \ge -L |x_i(t)- y(t)|$ in the previous relation, we obtain 
\begin{align*}
g_i(t) (y(t)-x^*)  \ge - 2L |y(t)-x_i(t)| 
+ f_i(y(t)) - f_i(x^*).
\end{align*}
Using the previous relation and $|g_i(t)| \le L$ in \eqref{eq:lemtemp1}, we have  
\begin{align*}
& \left( y(t+1)-x^* \right)^2 \le \left( y(t)-x^* \right)^2 + L^2  \alpha(t)^2  \\
& - \frac{2 \alpha(t)}{n} \sum_{i=1}^n \left( f_i(y(t))-f_i(x^*)\right) +  \frac{4 L  \alpha(t)}{n} \sum_{i=1}^n |y(t)-x_i(t)|.
\end{align*}
Rearranging the terms of the previous relation yields that 
\begin{align*}
& \frac{2 \alpha(t)}{n} \sum_{i=1}^n \left( f_i(y(t))-f_i(x^*) \right)\\
&  \le -\left(y(t+1)-x^* \right)^2 + \left( y(t)-x^* \right)^2 + L^2 \alpha(t)^2   \\
& +  \frac{4 L  \alpha(t)}{n} \sum_{i=1}^n |y(t)-x_i(t)|.
\end{align*}
Taking summation of the previous relation from $t=0$ to $T$ and noting the telescopic cancellation of terms, we obtain  
\begin{align*}
&\frac{2}{n} \sum_{t=0}^T \alpha(t) \left( F(\mathbf{y}(t))- F(\mathbf{x}^*)\right)  \\
& =  \sum_{t=0}^T \frac{2 \alpha(t)}{n} \sum_{i=1}^n \left( f_i(y(t))-f_i(x^*) \right)  \le \left(y(0)-x^* \right)^2 \\
& + L^2 \sum_{t=0}^T \alpha(t)^2  +   \sum_{t=0}^T \frac{4 L  \alpha(t)}{n} \sum_{i=1}^n |y(t)-x_i(t)|.
\end{align*}
Since $\hat{y}(T)= \frac{1}{\sum_{t=0}^T \alpha(t)} \sum_{t=0}^T \alpha(t) y(t)$. Using convexity of the functions along with the previous relation, we obtain \eqref{eq:lem:optimality2}, 
which completes the proof. 
\end{proof}
Next, we will bound the difference between the objective function value at $\{\hat{\mathbf{x}}(t)\}_{t=0}^{\infty}$ and the optimal point.
\begin{theorem}\label{th:optimality}
The ergodic average of the sequence generated by algorithm \eqref{eq:estimateupdate}-\eqref{eq:weightupdate}, satisfies 
\begin{align} \label{eq:optmalitybound}
&  F(\hat{\mathbf{x}}(T))- F(\mathbf{x}^*)  \le \frac{n}{2 \sum_{t=0}^T \alpha(t)} \left( y(0)-x^* \right)^2 \nonumber \\
& + \frac{n L^2}{2 \sum_{t=0}^T \alpha(t)} \sum_{t=0}^T {\alpha(t)^2} \nonumber \\
& + \frac{2 L}{ \sum_{t=0}^T \alpha(t)} \sum_{t=0}^T  \alpha(t) \sum_{i=1}^n |y(t)-x_i(t)|  + L \frac{1}{\sum_{t=0}^T \alpha(t)} \nonumber \\
& \left( \sum_{t=0}^T \alpha(t) \left( C \lambda^{t} ||\mathbf{x}(0)||_1 + L  n \sum_{s=0}^{t-1} C \lambda^{t-s} \alpha(s) \right) \right) 
\end{align}
\end{theorem}
\begin{proof}
Follows by combining Lemma \ref{lem:optimality1} and Lemma \ref{lem:optimality2}.
\end{proof}
\vspace{-.12 in}
\subsection{Convergence Rate: Choice of Step size}
In this section we characterize the convergence rate of the algorithm for $\alpha(t)=\frac{1}{\sqrt{t+1}}$, $t \ge 0$, as the step size sequence. We will show this choice of step size yields $O(\frac{\log T}{\sqrt{T}})$ convergence iteration complexity, which is the optimal convergence rate of any first-order optimization method over the class of convex functions as shown in \cite{nesterov2004introductory} (see Section 3 of \cite{nesterov2004introductory} for general lower complexity bounds). 
\begin{theorem}\label{th:ratestepsize}
Using the ergodic average of the sequence generated by the proposed algorithm with $\alpha(t)=\frac{1}{\sqrt{t+1}}$, both rate of convergence to the optimal solution and rate of convergence of estimates to a consensus point are $O(\frac{\log T}{ \sqrt{T}})$. In particular, with $\mathbf{x}(0)=0$, for all $i, j \in V$ we have  
\begin{align}\label{eq:feasibilityrate}
& |\hat{x}_i(T) - \hat{x}_j(T)|  \le  2 L n C \frac{4}{1-\lambda} \frac{\log T}{\sqrt{T}},
\end{align}
and 
\begin{align}\label{eq:optimalityrate}
&  F(\hat{\mathbf{x}}(T))- F(\mathbf{x}^*)  \le n \left( x^* \right)^2 \frac{1}{2\sqrt{T}}  + \frac{n L^2}{2} \frac{2\log T}{\sqrt{T}} \nonumber\\
& + 2  L^2 n C \frac{4}{1-\lambda} \frac{\log T}{\sqrt{T}} + n  L^2 C \frac{4}{1-\lambda} \frac{\log T}{\sqrt{T}}.
\end{align}
\end{theorem}
\begin{proof}
With $\alpha(t)=\frac{1}{\sqrt{t+1}}$, we have  
\begin{align*}
\sum_{t=0}^{T} \alpha(t) \ge \int_1^{T+2} \frac{1}{\sqrt{x}} \mathrm{d}x = 2(\sqrt{T+2}-1) \ge \sqrt{T}.
\end{align*}
Since $\alpha(t) \le 1$, we have $\sum_{t=1}^{T} \alpha(t) \lambda^t \le \frac{1}{1-\lambda}$, and 
\begin{align*}
& \sum_{t=0}^T \sum_{s=0}^{t-1} \lambda^{t-s-1} \alpha(s) \alpha(t) = 1+ \sum_{i=1}^{T} \lambda^{i-1} \left( \sum_{j=1}^{T-i+1} \frac{1}{\sqrt{j} \sqrt{i+j}} \right) \\
& \le 1+ \sum_{i=1}^{T} \lambda^{i-1} \int_{x=0}^{T-i+1} \frac{1}{\sqrt{x}\sqrt{x+i}}\\
& = 1+ \sum_{i=1}^{T} \lambda^{i-1} 2 \left( \log(\sqrt{T-i+1}+ \sqrt{T-1}) \right) \\
& \le 1+  2 \frac{1}{1-\lambda} \left( \log 2 + \frac{1}{2}\log T \right) \le \frac{4}{1-\lambda} \log T,
\end{align*}
for $T \ge 2$. 
Using the previous three relations in Theorem \ref{th:feasibility}, for any $i$ and $T$ we obtain 
\begin{align*}
& |\hat{x}_i(T) - \hat{y}(T)| \le C \frac{1}{1-\lambda} ||\mathbf{x}(0)||_1 \frac{1}{\sqrt{T}} +  L n C \frac{4}{1-\lambda} \frac{\log T}{\sqrt{T}}.
\end{align*}
This shows that rate of convergence to a consensus point  is $O(\frac{\log T}{ \sqrt{T}})$. In particular, using this relation for $i$ and $j$ and setting $\mathbf{x}(0)=0$, we obtain \eqref{eq:feasibilityrate}. 
Using the three relations along with Theorem \ref{th:optimality} and $\sum_{t=0}^T \alpha(t)^2 \le 1+ \log (T+1) \le 2 \log T$ (for $T \ge 4$), yields 
\begin{align*}
&  F(\hat{\mathbf{x}}(T))- F(\mathbf{x}^*)   \nonumber \\
& \le n \left( y(0)-x^* \right)^2 \frac{1}{2\sqrt{T}} + \frac{n L^2}{2} \frac{2\log T}{\sqrt{T}} \nonumber \\
& + 2 L C \frac{1}{1-\lambda} ||\mathbf{x}(0)||_1 \frac{1}{\sqrt{T}} + 2  L^2 n C \frac{4}{1-\lambda} \frac{\log T}{\sqrt{T}} \nonumber \\
& +  L C \frac{1}{1-\lambda} ||\mathbf{x}(0)||_1 \frac{1}{\sqrt{T}} + n  L^2 C \frac{4}{1-\lambda} \frac{\log T}{\sqrt{T}},
\end{align*}
This shows that the rate of convergence to the optimal solution is $O(\frac{\log T}{\sqrt{T}})$. In particular,  setting $\mathbf{x}(0)=0$, we obtain \eqref{eq:optimalityrate}. 
\end{proof}
\vspace{- .03 in}
\section{Numerical Results}\label{sec:numerical}
In this section, we show numerical results based on the proposed distributed subgradient algorithm to demonstrate the performance of the algorithm. We consider minimizing the function $F(x)=\frac{1}{2} \sum_{i=1}^n (x-a_i)^2$ where $a_i$ is a scalar that is known only to node $i$. This problem appears in distributed estimation where the goal is to estimate the parameter $x^*$, using local measurements $a_i=x^*+N_i$ at each node $i=1, \dots, n$. Here $N_i$ represents measurements noise, which we assume to be jointly Gaussian with mean zero and variance one.  The maximum likelihood estimate is the minimizer $x^*$ of $F(x)$. We let $n=20$, $a_i=i$, and consider a directed graph shown in Figure \ref{fig:simul}. We plot the differences $\hat{x}_1(T)-x^*$ and $\hat{x}_n(T)-x^*$ as a function of $T$ in Figure \ref{fig:simulexample}. 
\vspace{-.03 in}
\begin{figure}[t]
\centering
  \includegraphics[width=1 \linewidth]{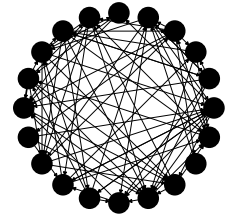}
  \caption{Sample Network with $n=20$ nodes. }
  \label{fig:simul}
\end{figure}
\vspace{- .05 in}
\begin{figure}[t]
\centering
  \includegraphics[width=1. \linewidth]{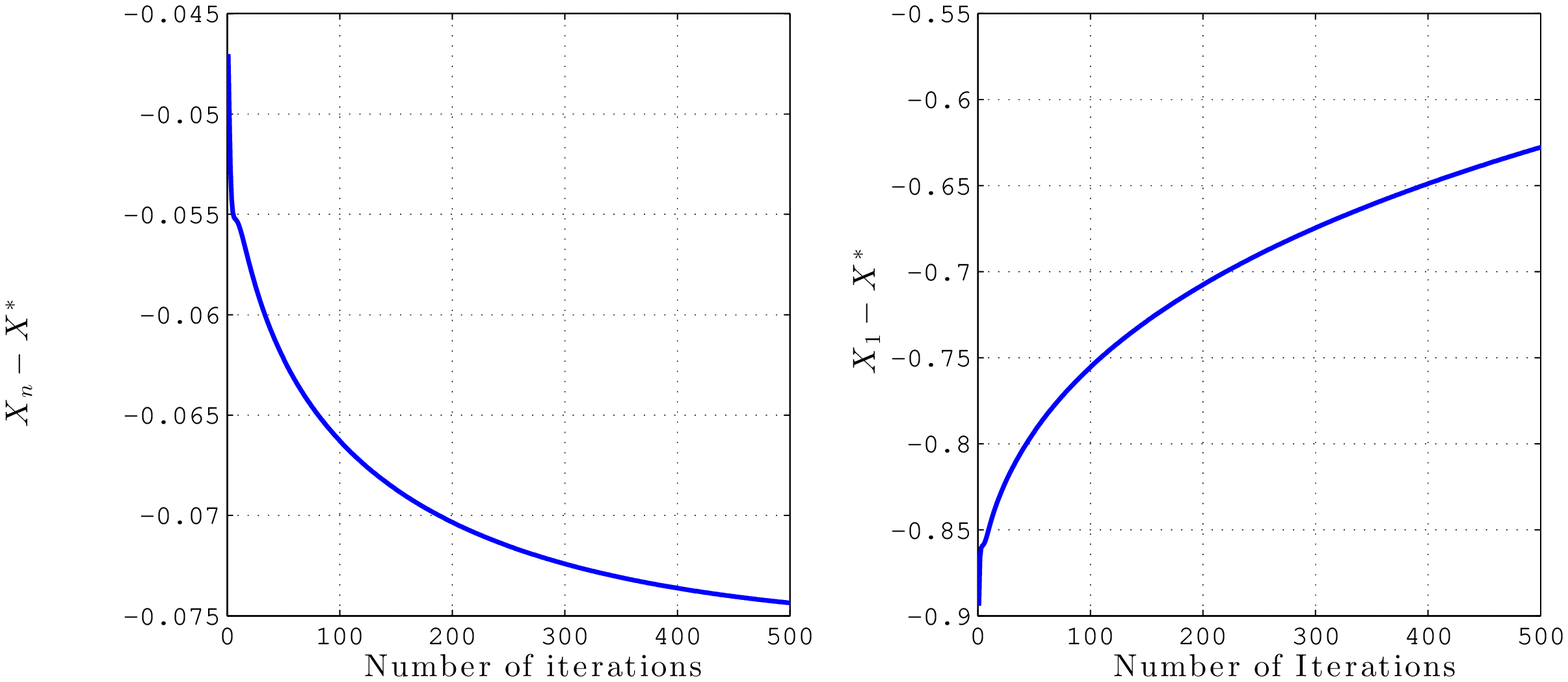}
  \caption{Performance of algorithm over sample network: $\hat{x}_i(T)-x^*$ for $i=1$ and $i=20$ as a function of $T$. }
  \label{fig:simulexample}
\end{figure}
\vspace{- .1 in}
\section{Conclusion}\label{sec:conclusion}
We considered a multi agent optimization problem where a directed network of agents collectively solves a global 
optimization problem with the objective function given by the sum of locally known convex 
functions. We propose a distributed subgradient algorithm in which each agent performs local processing based on his local information and information obtained from his incoming neighbors. In our algorithm, agents scale the  information they send to others with dynamically updated weights that converge to balancing weights of the graph and we show this technique overcomes the asymmetries caused by the directed communication network. We show that both the rate of convergence of the objective function to optimal value, and the rate of convergence of the estimates to a consensus point is $O(\frac{\log T}{\sqrt{T}})$, where $T$ is the number of iterations.

\bibliographystyle{abbrv}
\vspace{- .1 in}
\bibliography{references}
\end{document}